\documentclass[12pt]{amsart}
   \usepackage{amsmath,amsthm}
   \usepackage{amsfonts}   % if you want the fonts
   \usepackage{amssymb}    % if you want extra symbols
   \usepackage{url}

\advance \topmargin by -\headheight
\advance \topmargin by -\headsep
      
\evensidemargin \oddsidemargin
\usepackage[all]{xy}

\newtheorem{thm}{Theorem}[section]
\newtheorem{prop}[thm]{Proposition}
\newtheorem{lem}[thm]{Lemma}
\newtheorem{cor}[thm]{Corollary}
\newtheorem{conj}[thm]{Conjecture}
\newtheorem{defn}[thm]{Definition}
\newtheorem{qu}[thm]{Question}

\theoremstyle{remark}
\newtheorem{rem}[thm]{Remark}
\newtheorem{ex}[thm]{Example}

\title{On closed leaves of foliations, multisections and stable commutator lengths}
\author{Jonathan Bowden}
\address{Mathematisches Institut, Ludwig-Maximilians-Universit\"at, Theresienstr. 39, 80333 M\"unchen, Germany}
\curraddr{Mathematisches Institut, Universit\"at Augsburg, Universit\"atsstr. 14, 86159 Augsburg, Germany}
\email{jonathan.bowden@mathematik.uni-muenchen.de}
\date{\today}
\subjclass[2000]{Primary 57R30, 57R22; Secondary 57N13, 20F12, 20F69}
\begin{document}
\begin{abstract}
We give examples of foliations that answer two questions posed by Mitsumatsu and Vogt about the genus minimising properties of closed leaves of 2-dimensional foliations on 4-manifolds. By studying stable commutator lengths in certain stable mapping class groups, we also answer an asymptotic version of another question of theirs concerning bounds on self-intersection numbers of multisections in surface bundles.
\end{abstract}

\maketitle
\section{Introduction}
In \cite{MV} the authors study the problem of finding representatives for 2-dimensional homology classes that can be realised as leaves of foliations on a given 4-manifold. In dimension $4$ the existence of $2$-dimensional foliations reduces to the existence $2$-plane distributions, by an $h$-principle of Thurston. Moreover, the existence of an oriented $2$-dimensional distribution is equivalent to a splitting of the tangent bundle as the Whitney sum of two rank-2 subbundles. The Euler classes of these distributions must satisfy certain necessary conditions given by the Whitney sum formula. If an embedded surface $\Sigma$ can be realised as a leaf of a foliation on a $4$-manifold $M$, then the foliation defines a flat connection on the normal bundle $\nu_{\Sigma}$ of $\Sigma$ via the Bott construction. The Milnor inequality then implies that the Euler class of $\nu_{\Sigma}$ satisfies \[|e(\nu_{\Sigma})| \leq g(\Sigma) - 1.\] This inequality is then an obvious necessary condition for a given surface to be realisable as a leaf of a foliation. 

Using Thurston's $h$-principle Mitsumatsu and Vogt show in \cite{MV} that $\Sigma$ can be made a leaf of a foliation if and only if its normal bundle satisfies the Milnor inequality and there exist cohomology classes $e_1,e_2$ satisfying the necessary cohomological conditions mentioned above as well as the following equations: 
\[ e_1([\Sigma]) = 2 - 2g(\Sigma) \text{ and } e_2([\Sigma]) = e(\nu_{\Sigma}).[\Sigma] = [\Sigma]^2.\]
By using an alternate form of these cohomological criteria, we generate many examples of distributions where one has great flexibility in solving the additional equations needed to realise a given surface $\Sigma$ as a leaf. These examples then show that the genus of leaves representing a fixed homology class which can be realised as closed leaves is not unique and that the Euler class of a foliation does not determine whether the genera of leaves are minimal. This then yields negative answers to the two corresponding questions that were posed in \cite{MV} (cf.\ Examples \ref{foliated_genus} and \ref{leaf_minimise}).

Another rich source of foliations on 4-manifolds come as the horizontal foliations of flat surface bundles. Closed leaves of such foliations correspond to finite orbits under the holonomy given by the flat structure. Such leaves are examples of so-called $k$-multisections, where $k$ is the number of points in the intersection of the leaf with a fibre. In studying the self-intersection numbers of closed leaves of 2-dimensional foliations on 4-manifolds Mitsumatsu and Vogt posed the following question.
\begin{qu}[\cite{MV}, Problem 8.12]\label{Mit_Vogt_Prob}
For a given $h$ and $g$, does there exist an upper bound for the self-intersection number of any multisection of any $\Sigma_h$ bundle over $\Sigma_g$?
\end{qu}
An important initial observation for dealing with this question is the fact that the vertical Euler class of a surface bundle is bounded in the sense of Gromov (Proposition \ref{Euler_bounded}). We give two new proofs of this fact, which is originally due to Morita. The first of these relies on the adjunction inequality coming from Seiberg-Witten gauge theory and the second uses results of \cite{EnK} on the positivity of stable commutator lengths in the mapping class group. The boundedness of the vertical Euler class then implies that the self-intersection number of a section of a surface bundle is bounded in terms of the genus of the fibre and base. Similarly, the self-intersection number defines a characteristic class $e^v_k$ of $\Sigma_h$-bundles with $k$-multisections. This class is then also bounded and in this context one has a natural asymptotic analogue of Question \ref{Mit_Vogt_Prob} above. Namely, one would hope for a universal bound on the norms of the classes $e^v_k$.
\begin{qu}[Asymptotic version]\label{Mit_Vogt_Prob_asymptotic}
Are the norms $||e^v_k||_{\infty}$ bounded independently of $k$?
\end{qu}
\noindent This asymptotic version of the question of Mitsumatsu and Vogt is stronger than their original question and the main result of the second part of this paper is that it is false.
\begin{thm}
The sequence $||e^v_k||_{\infty}$ is unbounded.
\end{thm}
The proof of this fact exploits Bavard duality as well as results of \cite{Kot3} on the vanishing of the stable commutator lengths for certain stable groups. Moreover, this result fits with the observation that stable classes in stable groups tend to be unbounded and leads us to the conjecture that the bounded cohomology of the stable mapping class group is trivial.

\subsection*{Acknowledgments:}
The results of this article are taken from the author's doctoral thesis, which would not have been possible without the encouragement and support of Prof.\ D.\ Kotschick. The financial support of the Deutsche Forschungsgemeinschaft is also gratefully acknowledged.

\subsection*{Notation and Conventions:}
All bundles will be assumed to be oriented and all maps are smooth. Unless otherwise stated all homology groups will be taken with integral coefficients.

\section{Foliations with closed leaves on 4-manifolds}\label{distributions}
In general, a necessary condition for a manifold $M$ to admit a $k$-dimensional foliation is that it first admits a $k$-dimensional distribution. For 2-dimensional foliations of codimension greater than 1, this is however also sufficient.
\begin{thm}[\cite{Th1}, Cor.\ 3]\label{Th_h_principle}
Let $M$ be an oriented manifold of dimension greater than 3 and let $\xi$ be an oriented distribution of 2-planes. Then $\xi$ is homotopic to a foliation. Furthermore, if $\xi$ is integrable in a neighbourhood of a compact set $K \subset M$ then this homotopy may be taken relative to $K$.
\end{thm}
\noindent Moreover, the following result of Hirzebruch and Hopf expresses the existence of oriented distributions in terms of the existence of certain cohomology classes.
\begin{prop}[\cite{HH}]\label{char_sufficient}
A closed, oriented 4-manifold $M$ admits an oriented 2-plane distribution $\xi$ if and only if there exists a pair of characteristic elements for the intersection form $K_{+}, K_{-} \in H^2(M)$ such that:
\begin{equation*} \langle K_{\pm}^2, [M] \rangle = \pm 2\chi(M) + 3\sigma(M). \end{equation*}
Moreover, the Euler classes $e_1 = e(\xi)$ and $e_2 = e(\xi^{\perp})$ are given by the following formulae:
\[2e_1 = K_+ + K_{-} \ \]
\[2e_2 = - K_+ + K_{-} \ .\]
\end{prop}
\begin{rem}
Whilst the formula for the Euler classes of the distributions in Proposition \ref{char_sufficient} is not explicitly given in \cite{HH}, it follows readily from their arguments and we refer to (\cite{Bow2}, Section 2.2) for details.
\end{rem}
%In \cite{MV} the authors use a different form of these cohomological equations, which we call the distribution equations and which we note for future reference.
%\begin{prop}[Distribution equations, \cite{MV}]\label{Distribution_equations}
%Let $M$ be a 4-manifold and let $Tor_2 \subset H^2(M)$ be the subgroup of 2-torsion elements. There exist complementary 2-plane distributions $\xi$ and $\xi^{\perp}$ on $M$ with Euler classes $e_1,e_2 \in H^2(M) / Tor_2$ if and only if the following equations hold:
%\[\langle e_1^2 + e_2^2, [M] \rangle = 3\sigma(M) \]
%\begin{equation}\label{distributioneq_1}  \langle e_1 \smallsmile e_2, [M] \rangle = \chi(M)  \end{equation}
%\[e_1 + e_2 \equiv w_2(M) \text{ mod 2}.\]
%\end{prop}
Using Proposition \ref{char_sufficient} one may give necessary and sufficient conditions for the existence of distributions in terms the Euler characteristic and signature of $M$. In particular, the existence of a smooth foliation only depends on the homotopy type of $M$. The following result goes back to Atiyah and Saeki (cf.\ \cite{Mats}, Theorem 2).
\begin{prop}[Existence of distributions]\label{existence_distributions}
Let $M$ be an oriented 4-manifold with indefinite intersection form, then $M$ admits a distribution if and only if
\[ \sigma(M) \equiv 0 \text{ mod 2} \emph{ and } \chi(M) \equiv \sigma(M) \text{ mod 4}. \] 
\end{prop}
%The Milnor inequality expresses a relationship between the Euler class of certain flat bundles over closed surfaces and the genus of the base. This inequality is due to Milnor in the case of $GL^+(2,\mathbb{R})$-bundles and its generalisation to $Homeo^+(S^1)$ is due to Wood. We shall in fact only need Milnor's original inequality in what follows, but in accordance with standard usage we will refer to the following result as the Milnor inequality.
%\begin{thm}[Milnor inequality, \cite{Mil}]\label{Milnor_Wood}
%A $GL^+(2,\mathbb{R})$-bundle $E \to \Sigma$ over a surface of genus $g > 0$ is flat if and only if
%\[ |e(E)| \leq g(\Sigma) -1.\]
%\end{thm}
In \cite{MV} the authors give criteria for the existence of 2-dimensional foliations with closed leaves on a 4-manifold $M$. The basic observation is that the Milnor inequality (cf.\ \cite{Mil}) puts homological restrictions on which classes $[\Sigma] \in H_2(M)$ can occur as leaves of a foliation. To begin with we note that if $\Sigma$ is a leaf of some foliation $\mathcal{F}$ on $M$, then there is a connection on the normal bundle $\nu(\mathcal{F})$ that is flat when restricted to leaves of the foliation. This is a so-called Bott connection (see \cite{Bott}). Thus, in particular, if $\Sigma$ is a leaf of some foliation then the normal bundle of $\Sigma$ is a flat bundle and the Milnor inequality implies that
\[ |[\Sigma]^2| = |e(\nu(\mathcal{F})).[\Sigma]| \leq g(\Sigma) -1.\]
Conversely, there are sufficient conditions for a given embedded surface to be a leaf of some foliation on a 4-manifold $M$. These are given in the following theorem of Mitsumatsu and Vogt.
\begin{thm}[\cite{MV}, Th.\ 4.4]\label{Mits_Vogt}
If a compact surface $\Sigma$ in a 4-manifold $M$ satisfies the Milnor inequality, then $\Sigma$ can be realised as a leaf of a foliation $\mathcal{F}$ if and only if there is a distribution $\xi$ such that the Euler classes $e_1 = e(\xi)$ and $e_2 = e(\xi^{\perp})$ satisfy
\begin{equation}\label{leaf_eq}
\langle e_1,[\Sigma] \rangle = \chi(\Sigma), \ \ \ \langle e_2,[\Sigma] \rangle =  [\Sigma]^2.
\end{equation}
\end{thm}
%\begin{rem}\label{Mits_Vogt_rem}
%The proof of Theorem \ref{Mits_Vogt} in \cite{MV} shows that if $\xi_1, \xi_2$ are $2$-plane fields whose Euler classes $e_1 = e(\xi_1)$ and $e_2 = e(\xi_2)$ satisfy the conditions of Theorem \ref{Mits_Vogt}, then there is a foliation having $\Sigma$ as a leaf that is homotopic to $\xi_1$ as a distribution.
%\end{rem}
Mitsumatsu and Vogt posed several questions about the properties of surfaces that are leaves of foliations. In particular, they asked whether a class $\sigma \in H_2(M)$ knows its foliated genus (\cite{MV}, Question 8.8), that is if $\Sigma_1, \Sigma_2$ are leaves of foliations $\mathcal{F}_1, \mathcal{F}_2$  and $[\Sigma_1] = [\Sigma_2]$ in $H_2(M)$, does it then follow that $\chi(\Sigma_1) = \chi(\Sigma_2)$. The following examples provide a negative answer to this question.
\begin{ex}[The genus of leaves representing a fixed homology class]\label{foliated_genus}
Suppose we have a manifold $M$ whose signature and Euler characteristic satisfy the congruences of Proposition \ref{existence_distributions} so that $M$ admits distributions. Assume further that the intersection form on $M$ is odd and is of the form
\[Q \cong 2 \langle 1 \rangle \oplus 2 \langle -1 \rangle.\]
For example one can take 
\[M = 2 \mathbb{C}P^2 \# \thinspace 2 \overline{\mathbb{C}P}^2  \#  \thinspace  S^1 \times S^3.\]
We choose a basis $v_1,...v_4$ of $H^2(M)$ so that the intersection form $Q$ has the stipulated from. One then has an explicit family of solutions of the equations in Proposition \ref{char_sufficient} given by setting
\[ K_{+} = (2t - s, 2y + 1,  2y - 1, 1)\]
\[ K_{-} = (2t + s,  2z + 1,  2z - 1, 1).\]
Here $s$ and $t$ are free integral parameters, $s$ is odd and $y,z$ are chosen so that the equations in Proposition \ref{char_sufficient} are satisfied. Then the corresponding Euler classes are
\begin{align*}
e_1 & = (2t, y +z + 1, y+z - 1,  1)\\
e_2 & = (s, z - y , z -  y,  0).
\end{align*}
We let $[\Sigma] $ be a surface of minimal genus representing the class $v_1$ so that $[\Sigma]^2 = 1$ and set $s = 1$ so that the second equation in Theorem \ref{Mits_Vogt} above is satisfied. Then as $t$ was a free parameter in the set of solutions, it may be chosen so that the first equation in Theorem \ref{Mits_Vogt} is also satisfied. Furthermore, we may glue in as many trivial handles as we like to obtain a surface $\widetilde{\Sigma}$ that is homologous to $\Sigma$ and has arbitrarily large genus. In particular, we can assume that the Milnor inequality is satisfied and we see that $v_1$ has representatives of infinitely many genera that can be made a leaf of a foliation. It is easy to see that this example generalises to manifolds $M$ with odd intersection form and $b_2^{\pm}(M) \geq 2$, one simply splits the intersection form
\[ Q \cong 2\langle 1 \rangle \oplus 2\langle -1 \rangle \oplus \overline{Q}\]
and augments the solutions for $(K_{+}, K_{-})$ given above by putting odd integers in the remaining entries.
In this more general case we may even assume that $M$ is simply connected by setting
\[M = k\mathbb{C}P^2 \# l\overline{\mathbb{C}P}^2,\]
where $k$ and $l$ are both odd and larger than $3$.
\end{ex}
Another question posed by Mitsumatsu and Vogt is whether the Euler class of a distribution determines whether or not every leaf of a foliation with this Euler class will be genus minimising or not. This is false as the following example shows, providing a negative answer to Question 8.7 of \cite{MV}.
\begin{ex}[The Euler class does not determine whether the genera of leaves are minimal]\label{leaf_minimise}
We let $M = T^2 \times \Sigma_g$ and endow it with a product symplectic structure so that the two factors are symplectic submanifolds. Then the homology class $\sigma = [T^2 \times pt] + [pt \times \Sigma_g] = T + S $ can be represented by a symplectic surface $\Sigma$ of genus $g + 1$ that one obtains by resolving the intersection point of $(T^2 \times pt) \bigcup (pt \times \Sigma_g)$. Moreover, by the symplectic Thom conjecture this surface is genus minimising (cf.\ \cite{OS}) and the Milnor equality is satisfied if $g \geq 3$. 

The intersection form $M$ is isomorphic to $(2g + 1) H$ and we may take a hyperbolic basis $\{e_i\}$ for $H^2(M)$ with $e_1 = T$, $e_2 = S$. Since $\chi(M) = \sigma(M) = 0$ we have the following solutions of the equations in Proposition \ref {char_sufficient}:
\begin{align*}
 K_{+} &= (-2 - 2g, 0,  -2 - 2g',0,..., 0)\\
K_{-} &= (2 - 2g, 0,  2 - 2g',0,..., 0),
\end{align*}
and the corresponding Euler classes are
\begin{align*}
e_1 & = (-2g, 0, -2g',  0,..,0)\\
e_2 & = (2, 0, 2,0,...,0).
\end{align*}
Thus if we let $\Sigma'$ be a surface of genus $g' +1$ representing the class $\sigma' = (0,0,1,1,0,..,0)$ with $g'$ large, then the Milnor inequality and equations (\ref{leaf_eq}) hold for both $\Sigma$ and $\Sigma'$. However, $\Sigma$ is genus minimising and $\Sigma'$ is not. Hence the Euler class of a foliation does not determine whether closed leaves of a foliation with the given Euler class is genus minimising or not.
\end{ex}

\section{Surface bundles and boundedness of the vertical Euler class}\label{surface_bundles}
We begin by recalling some generalities about surface bundles and their sections, which for the most part can be found in \cite{Mor}. Let $\Gamma_h = Diff^+(\Sigma_h)/ Diff_0(\Sigma_h)$ denote the mapping class group of an oriented Riemann surface $\Sigma_h$ of genus $h$. By the classical result of Earle and Eells the identity component $Diff_0(\Sigma_h)$ is contractible in the $C^{\infty}$-topology if $ h \geq 2$ (cf.\ \cite{EE}). Thus the classifying space $B Diff^+(\Sigma_h)$ is homotopy equivalent to $B \Gamma_h$, which is in turn the Eilenberg-MacLane space $K(\Gamma_h,1)$.

In general, any bundle is determined up to bundle isomorphism by the homotopy class of its classifying map and since $B Diff^+(\Sigma_h)$ is aspherical, a surface bundle $\Sigma_h \to E \to B$ is determined up to bundle isomorphism by the conjugacy class of its \emph{holonomy representation}:
\[ \rho: \pi_1(B) \to \Gamma_h.\]
The conjugation ambiguity is a result of the choice of base points.
Conversely, any homomorphism $\rho: \pi_1(B) \to \Gamma_h$ induces a map $B \to K(\Gamma_h,1) = B \Gamma_h$ and thus defines a bundle that has holonomy $\rho$.

A section of a bundle $E$ is equivalent to a lift of the holonomy map of the bundle to $\Gamma_{h, 1}$, which denotes the mapping class group of $\Sigma_h$ with one marked point. That is $E$ admits a section if and only if there is a lift $\bar{\rho}$ so that the following diagram commutes
\[\xymatrix{ & \Gamma_{h,1} \ar[d] \\
\pi_1(B) \ar[r]^{\rho} \ar@{-->}[ur]^{\bar{\rho}} & \Gamma_h.} \]
Similarly, a $k$-multisection is the same as a lift of the holonomy map to the mapping class group with $k$ marked points that we denote by $\Gamma_{h, k}$. Here we require only that the set of $k$ marked points be fixed \emph{as a set}, rather than that each marked point itself be fixed by elements of $\Gamma_{h, k}$.

%From any bundle with a $k$-multisection one obtains in a natural way one with a section. For if $S \hookrightarrow E$ is a multisection and if we denote by $p$ the composition of this inclusion with the projection to the base $B$, then the pullback bundle $p^* E$ has a natural section $\widetilde{S}$ induced by $S$.

As in the case of ordinary surface bundles, the classifying space of surface bundles with a section is the Eilenberg-MacLane space $B \Gamma_{h,1} = K(\Gamma_{h,1}, 1)$, if $h \geq 2$. Furthermore, there is a natural exact sequence given by forgetting the marked point
\[1 \to \pi_1(\Sigma_h) \to \Gamma_{h,1} \to \Gamma_h \to 1\]
and one may identify $B \Gamma_{h,1}$ with the the total space $E \Gamma_h$ of the universal bundle over $B \Gamma_h$. Similarly, the classifying space for bundles with a $k$-multisection is the Eilenberg-MacLane space $B \Gamma_{h,k} = K(\Gamma_{h,k}, 1)$.

The bundle of vectors that are tangent to the fibres of the projection $E \Gamma_h \to B \Gamma_h$ is an oriented rank-2 vector bundle. The Euler class of this bundle defines a cohomology class $e \in H^2(E \Gamma_h) = H^2(B \Gamma_{h,1})$, which we will call the \emph{vertical Euler class}. Alternately, one can define $e$ as the Euler class associated to the central extension
\[1 \to \mathbb{Z} \to \Gamma^1_h \to \Gamma_{h,1} \to 1,\]
where $\Gamma^1_h = Diff^c(\Sigma^1_h)/ Diff^c_0(\Sigma^1_h)$ denotes the compactly supported mapping class group of a once punctured, genus $h$ surface. Here the right most map is given by collapsing the boundary to a point and the kernel is generated by a Dehn twist along a curve parallel to the boundary. Under the identification of $E \Gamma_h$ with the classifying space $B\Gamma_{h,1}$ of bundles with a section $S$, the class $e$ corresponds to the characteristic class given by restricting the vertical Euler class to $S$. Furthermore, if the base of the bundle is a surface, this is just the self-intersection number of the section.

By considering the diagonal section $\Delta_h \subset \Sigma_h \times \Sigma_h$ one sees that the class $e \in H^2(E \Gamma_h) =  H^2(\Gamma_{h,1})$ is non-trivial for $h \geq 2$. This example can be somewhat generalised by considering certain coverings (compare \cite{MV}, Example 8.13).
\begin{ex}[Multisections in $\Sigma_h \times \Sigma_g$]\label{sections_of_bundles}
Let $g,h \geq 2$ and let $\Sigma = \Sigma_2$. We then choose a surjective homomorphism $\pi_1(\Sigma) \to \mathbb{Z}_{g-1} \times \mathbb{Z}_{h-1}$. Denote by $\widetilde{\Sigma}$ the covering space associated to this homomorphism and let $G_1 = \mathbb{Z}_{g-1} \times 0$ and $G_2 = 0 \times \mathbb{Z}_{h-1}$ be the two factors of $G = \mathbb{Z}_{g-1} \times \mathbb{Z}_{h-1}$. By construction $G_1$ and $G_2$ act freely on $\widetilde{\Sigma}$ and we thus obtain coverings $\widetilde{\Sigma} \stackrel{\pi_i} \rightarrow \widetilde{\Sigma} / G_i$. We define a map
\[\widetilde{\Sigma} \stackrel{\pi_1 \times \pi_2} \rightarrow (\widetilde{\Sigma} / G_1) \times (\widetilde{\Sigma} / G_2) = \Sigma_h \times \Sigma_g.\]
One checks that this map is transverse to the fibres $\Sigma_h$, so that its image defines a multisection $S$ of $E$ and that
\[ |[S]^2| = |(2-2h).(g-1)| = |\chi(S)|.\]
\end{ex}
\noindent The diagonal section $\Delta_h$ has the property that $|[\Delta_h]^2| = 2h-2$. For fixed genus this self-intersection number is in fact maximal as the following proposition shows.
\begin{prop}\label{section_bound}
Let $S$ be a section of a surface bundle $\Sigma_h \to E \to \Sigma_g$ with $h \geq 2$, then $|[S]^2| \leq \text{max} \thinspace \{0,2g-2\}$.
\end{prop}
\begin{proof}
We first consider the case $g \geq 2$. Since the genus of the fibre is greater than 1, we may apply the Thurston construction to obtain a symplectic form on $E$ with respect to which $S$ is symplectic \emph{and} $[S]^2 \geq 0$. Because the genus of fibre and base are positive, the bundle $E$ is aspherical and it follows from Proposition 1 of \cite{Kot1} that
\[|\sigma(E)| \leq \chi(E).\]
This gives a lower bound for $b^+_2(E)$ and $b^-_2(E)$, since
\[2\text{min}\{b^+_2(E), b^-_2(E)\} = b_2(E) - |\sigma(E)|\]
and
\[ b_2(E) - |\sigma(E)| = \chi(E) - 2 + 2b_1(E) - |\sigma(E)| \geq 2b_1(E) - 2 \geq 6.\]
Thus in either orientation $b_2^+(E) \geq 3$ and by applying the adjunction inequality (cf.\ \cite{GS}) we conclude that
\begin{equation*}
\label{adj_inq}  |[S]^2| \leq |[S]^2| + |K.[S]| \leq 2g(S) - 2.
\end{equation*}
If $g < 2$ then as both $S^2$ and $T^2$ have self maps of arbitrary degree we may assume after taking coverings that $|[S]^2|$ is large, say $|[S]^2|  > 3$ or that $|[S]^2|=0$ in which case the inequality is trivially true. In the former case we take the pullback bundle under a degree 1 collapsing map from a genus $2$ surface $\Sigma_2$ to $\Sigma_g$. The resulting bundle has a section $S'$ of self-intersection 
\[|[S']^2| > 3 > 2g(\Sigma_2) - 2,\]
which yields a contradiction.
\end{proof}
%We contrast the bound given by Proposition \ref{section_bound} with that given by the Milnor inequality which provides a better bound for the self-intersection number. In particular, we see that the multisections in Example \ref{sections_of_bundles} cannot be realised as a leaf of a foliation as they do not satisfy the Milnor inequality.
As a consequence of Proposition \ref{section_bound} we will prove that the class $e \in H^2(\Gamma_{h,1})$ is bounded in the sense of Gromov. This fact is originally due to Morita, whose proof uses the boundedness of the Euler class in $Homeo^+(S^1)$ as proved by Wood (cf.\ \cite{Mor3}). On the other hand we will present two alternate proofs, the first of which relies on Proposition \ref{section_bound}, which is in turn a consequence of the adjunction inequality, whilst the second uses results of \cite{EnK}. Before giving these proofs we first recall the definition of the Gromov-Thurston (semi-)norm.
\begin{defn}[Gromov-Thurston norm]
Let $X$ be any topological space and let $\alpha \in H_2(X, \mathbb{Z})$. Define the minimal genus $g_{min}(\alpha)$ of $\alpha$ to be the minimal $g$ so that $\alpha$ is representable as the image of the fundamental class under a map $\Sigma_g \to X$ modulo torsion. We define the Gromov-Thurston norm to be
\[||\alpha||_{GT} = \lim\limits_{n \to \infty} \frac{2g_{min}(n \alpha) - 2}{n}.\]
\end{defn}
\noindent Gromov has also defined another norm on real homology. This is the so-called $l^1$-norm and is defined as follows.
\begin{defn}[$l^1$-norm]
Let $ c = \sum_i \lambda_i \sigma_i$ be a chain in $C_k(X, \mathbb{R})$. We define the $l^1$-norm of $c$ to be
\[ || c ||_1 =  \sum_i |\lambda_i|.\]
For a class $\alpha \in H_k(X, \mathbb{R})$ we define
\[ ||\alpha||_1 = \inf \ \{ \  || z||_1 \ | \ \alpha = [z] \}. \]
\end{defn}
 In degree 2 the Gromov-Thurston norm agrees with the $l^1$-norm $|| \cdot ||_1$ up to a constant. We record this in the following lemma, a proof of which is given in \cite{BGh}.
\begin{lem}\label{Th_Gr}
Let $\alpha \in H_2(X, \mathbb{Z})$. Then the following equality holds
\[ || \alpha ||_1 = 2 || \alpha ||_{GT}.\]
\end{lem}
If in addition the group $H_2(X, \mathbb{Z})$ is finitely generated, we may extend $|| \cdot ||_{GT}$ to a semi-norm on $H_2(X, \mathbb{R})$ and the equality of Lemma \ref{Th_Gr} holds for this extension. We shall denote this extension again by $|| \cdot ||_{GT}$.
\begin{lem}\label{Th_Gr2}
Let $H_2(X, \mathbb{Z})$ be finitely generated. Then there is a unique extension of $|| \cdot ||_{GT}$ to $H_2(X, \mathbb{R})$, for which $|| \cdot ||_1 = 2 || \cdot ||_{GT}$.
\end{lem}
By considering the natural pairing between homology and cohomology, one obtains a norm on cohomology that is dual to the $l^1$-norm. This norm agrees with the $l^{\infty}$-norm as introduced by Gromov (cf.\ \cite{Gro}). We then say that a cohomology class is \emph{bounded}, if it is bounded with respect to the $l^{\infty}$-norm. One may reformulate this definition in a slightly different manner. For this we let $H^*_b(X)$ denote the \emph{bounded cohomology}  of $X$, which is defined as the cohomology of the complex of $l^{\infty}$-bounded singular cochains. Then a cohomology class is bounded if and only if it lies in the image of the natural comparison map
\[H_b^*(X) \to H^*(X).\] We may now prove the boundedness of the vertical Euler class.
\begin{prop}[Morita]\label{Euler_bounded}
Let $h \geq 2$, then the vertical Euler class $e \in H^2(\Gamma_{h,1})$ is bounded and $|| e||_{\infty} = \frac{1}{2}$.
\end{prop}
\begin{proof}
For any natural number $n \in \mathbb{N}$ suppose that $ n\alpha \in H_2(\Gamma_{h,1})$ can be represented as the image of the fundamental class under a map $\Sigma_{n \alpha} \to B\Gamma_{h,1}$ and assume that $\Sigma_{n \alpha}$ is genus minimising. This in turn corresponds to a surface bundle with a section $S_{n\alpha}$ and by Proposition \ref{section_bound} we have that
\begin{equation}\label{adj_euler} 
 n \thinspace |e(\alpha)| = |e(n\alpha)| = |[S_{n\alpha}]^2| \leq 2g(\Sigma_{n \alpha}) - 2 = 2g_{min}(n \alpha) - 2
\end{equation}
and thus
\[|e(\alpha)| \leq \lim\limits_{n \to \infty} \frac{2g_{min}(n \alpha) - 2}{n} = || \alpha ||_{GT}.\]
Hence for integral classes $\alpha \in H_2(X, \mathbb{Z})$ Lemma \ref{Th_Gr} yields
\[ \frac{|e(\alpha)|}{|| \alpha ||_1} \leq \frac{1}{2}.\]
Now since the group $\Gamma_{h,1}$ is finitely presented (cf.\ \cite{Iva}, Theorem 4.3 D), it follows that $H_2(\Gamma_{h,1})$ is finitely generated. Thus we may extend the Gromov-Thurston norm to real cohomology by Lemma \ref{Th_Gr2} and 
\[\sup_{\alpha \neq 0} \frac{|e(\alpha)|}{|| \alpha ||_1} \leq \frac{1}{2}.\]
It follows that $e$ is a bounded class with $||e||_{\infty} \leq \frac{1}{2}$. We will show that this bound is sharp. 

Let $h \geq 2$ and let $S$ be the diagonal section of $\Sigma_h \times \Sigma_h$. This section has self-intersection $2 -2h$ and we denote the corresponding class $\alpha \in H_2(\Gamma_{h,1})$. By taking coverings we obtain sections $S_{n\alpha}$ so that
\[  |[S_{n\alpha}]^2| = |e(n \alpha)| = n (2h -2) = 2g(S_{n\alpha}) - 2.\]
Thus it follows that inequality (\ref{adj_euler}) is an equality for this $\alpha$ and hence
\[||e||_{\infty} = \sup_{ \alpha \neq 0} \frac{|e(\alpha)|}{|| \alpha ||_1} \geq \frac{1}{2}.\]
Combining this with our previous estimate gives the result.
\end{proof}
One may also give another proof of the boundedness of $e$ using results on stable commutator lengths of Dehn twists as proven in \cite{EnK}. Here as usual the commutator length $c_G(g)$ of an element $g$ in a group $G$ is defined as the smallest natural number so that $g$ can be written as a product of $c_G(g)$ commutators. The \emph{stable commutator length} is then defined as
\[|| g ||_{com} = \lim\limits_{n \to \infty} \frac{c_G(g^n) }{n}.\]
There is a relationship between stable commutator lengths and quasi-homomorphisms. We recall that a real valued map $\phi$ on a group $G$ is a quasi-homomorphism if the following supremum is finite
\[D(\phi) = \sup |\phi(gh) - \phi(g) - \phi(h)|.\]
The number $D(\phi)$ is called the \emph{defect} of $\phi$. A quasi-homomorphism is homogeneous if $\phi(g^n) = n \phi(g)$ for all integers $n$. For a group $G$ we let $\widetilde{QH}(G)$ denote the group of homogeneous quasi-homomorphisms. With this notation we have the following fundamental result of Bavard (cf.\ also \cite{Cal}).
\begin{thm}[\cite{Bav}]\label{length_Bavard}
For any element $g$ in a group $G$ the following holds
\[|| g ||_{com} = \sup \frac{|\phi(g)|}{2D(\phi)} \thinspace ,\]
where the supremum is taken over all $\phi \in \widetilde{QH}(G)$ and $D(\phi)$ is the defect of $\phi$.
\end{thm}
The relationship between stable commutator lengths and the boundedness of elements in the group cohomology $H^2(G)$ is given in the following lemma.
\begin{lem}\label{bound_commutator}
Let $\phi$ be a non-torsion element in $H^2(G)$, which corresponds to a central $\mathbb{Z}$-extension
\[ 1 \to \mathbb{Z} \to \hat{G} \stackrel{\pi} \rightarrow G \to 1.\]
Further, let $\Delta$ denote a generator of the kernel. Then $\phi$ is bounded if and only if the stable commutator length of $\Delta$ is positive.
\end{lem}
\begin{proof}
Applying the Bavard sequence, we obtain the following diagram
\[\xymatrix{0 \ar[r] & H^1(G,\mathbb{R}) \ar[r] \ar[d] &\widetilde{QH}(G) \ar[r] \ar[d] & H_b^2(G) \ar[r] \ar[d]^{\pi^*} & H^2(G,\mathbb{R})\ar[d]\\
0 \ar[r] & H^1(\hat{G},\mathbb{R}) \ar[r]  &\widetilde{QH}(\hat{G}) \ar[r]  & H_b^2(\hat{G}) \ar[r]& H^2(\hat{G},\mathbb{R}). }\]
The fact that $\mathbb{Z}$ is amenable implies that the map $H_b^2(G) \stackrel{\pi^*} \rightarrow H_b^2(\hat{G})$ is an isomorphism. Furthermore, since the Euler class of the extension $e \in H^2(G)$ spans the kernel of the map to $H^2(\hat{G},\mathbb{R})$, exactness implies that $e$ is bounded if and only if there is element $\phi_e \in \widetilde{QH}(\hat{G})$ that is mapped to $e$ under the composition of the comparison map to $H^2(G,\mathbb{R})$ with the inverse of $\pi^*$. Since any homogeneous quasi-homomorphism that vanishes on $\mathbb{Z}$ descends to the quotient, we deduce that $\phi_{e}$ does not vanish on $\mathbb{Z}$. Hence by Theorem \ref{length_Bavard} we conclude that $e$ is bounded if and only if the stable commutator length of $\Delta$ is positive.
\end{proof}
In order to deduce the boundedness of $e$ from Lemma \ref{bound_commutator} we consider its defining central $\mathbb{Z}$-extension
\[1 \to \mathbb{Z} \to \Gamma^1_h \to \Gamma_{h,1} \to 1.\]
The kernel of this extension is generated by a Dehn twist $\phi_C$ around a curve that is parallel to $C = \partial \Sigma^1_h$. Now the group $\Gamma^1_h$ embeds into $ \Gamma_{h+1}$ and the image of $\phi_C$ is a Dehn twist around a homotopically non-trivial, separating curve in $\Sigma_{h+1}$. By \cite{EnK} there is a positive lower bound on the stable commutator length of $\phi_C$ considered as an element in $ \Gamma_{h+1}$ and hence the same holds in $\Gamma^1_h$. Thus Lemma \ref{bound_commutator} implies that $e$ is bounded.

\section{Bounds on self-intersection numbers of multisections}
The self-intersection of a $k$-multisection in a surface bundle gives a characteristic class $e^v_k \in H^2(\Gamma_{h,k})$.  As mentioned in the introduction Mitsumatsu and Vogt have asked whether there is a universal bound on self-intersection numbers of multisections in surfaces bundles with fibre and base of fixed genus (cf.\ Question \ref{Mit_Vogt_Prob}). Or in other words, they ask if there exists a constant $C(h,g)$ so that
\[\sup_{k \in \mathbb{N}}  |e^v_k(E)| \leq C(h,g) < \infty,\]
where the supremum is taken over all bundles with a $k$-multisection that have fibre and base of genus $h$ and $g$ respectively. An initial observation is that the classes $e^v_k$ are bounded and that there is an obvious upper bound for $||e^v_k||_{\infty}$.
\begin{prop}\label{Euler_bounded_k}
Let $h \geq 2$, then the classes $e^v_k \in H^2(\Gamma_{h,k})$ are bounded and $||e^v_k||_{\infty} \leq \frac{k}{2}$.
\end{prop}
\begin{proof}
Since $\Gamma_{h,k}$ is finitely presented (cf.\ \cite{Iva}, Theorem 4.3 D) we may apply the first part of the proof of Proposition \ref{Euler_bounded} \emph{mutatis mutandis} to conclude that
\[ ||e^v_k||_{\infty} \leq \frac{k}{2}. \qedhere \]
\end{proof}
If one considers only multisections that are \emph{pure} in the sense that the multisection in question consists of $k$ disjoint sections, then one obtains a universal bound $C(h,g)$. A bundle that has a pure $k$-multisection is given by a holonomy map $\pi_1(\Sigma_g) \to P \Gamma_{h,k}$, where the $P$ means that the holonomy maps fix the marked points pointwise, then it follows from Proposition \ref{section_bound} that for fixed $g$ and $h$ the self-intersection number of any pure $k$-multisection is bounded.
\begin{prop}\label{pure_bound}
Let $\mathcal{C}_{pure}$ denote the class of all bundles with a pure multisection and let $g$ and $h$ be the genus of the base and fibre of $E$ respectively with $h \geq 2$. Then the following holds:
\[\sup_{k  \in \mathbb{N}, E \in \mathcal{C}_{pure}} |e^v_k(E)| \leq (2g - 2) (4gh +2).\]
\end{prop}
\begin{proof}
First note that the number of sections that can have non-zero self-intersection is bounded by $b_2(E) \leq 4gh +2$ and this does not depend on $k$. By Proposition \ref{section_bound} each section has self-intersection at most $2g - 2$. Thus if the holonomy of a bundle is in $P \Gamma_{h,k}$ we conclude that
\[\sup_{k  \in \mathbb{N}, E \in \mathcal{C}_{pure}} |e^v_k(E)| \leq (2g - 2) (4gh +2). \qedhere \]
\end{proof}
Unfortunately the bound obtained in Proposition \ref{pure_bound} grows quadratically in $g$ and thus gives no bound on the norms $||e^v_k||_{\infty}$, which would in particular yield an affirmative answer to Question \ref{Mit_Vogt_Prob}. For given any bundle with a $k$-multisection and holonomy map $\pi_1(\Sigma_g) \stackrel{\rho} \rightarrow \Gamma_{h,k}$, there is a finite cover $\Sigma_{\bar{g}}  \stackrel{\tau} \rightarrow \Sigma_g$ of degree $N = k!$ and a lift $\bar{\rho}$ so that the following diagram commutes
\[\xymatrix{ P \Gamma_{h,k} \ar[r] & \Gamma_{h,k}\\
\pi_1(\Sigma_{\bar{g}}) \ar[r]^{\tau} \ar[u]^{\bar{\rho}} &  \pi_1(\Sigma_g) \ar[u]^{\rho}.}\]
Then if $E, \bar{E}$ denote the bundles associated to $\rho, \bar{\rho}$ respectively, we have
\[|N e^v_k(E)| = |e^v_k(\bar{E})| \leq (2\bar{g} - 2) (4\bar{g}h +2) = N (2g -2)(4(N(g-1) + 1)h + 2).\]
Thus, the bound we obtain on $e^v_k(E)$ in this way depends on $N = k!$ which in turn grows exponentially in $k$.
%\subsection{Asymptotic bounds on self-intersection numbers of multisections} 

Instead of asking for a bound on the self-intersection numbers of multisections in surface bundles with fixed base and fibre, one may ask for a universal bound on the norms $||e^v_k||_{\infty}$ that is independent of $k$, as opposed to the bound given by Proposition \ref{Euler_bounded_k}, which grows linearly with $k$. This question may in turn be viewed as a stable or asymptotic version of the original question of Mitsumatsu and Vogt. Of course, it is strictly stronger than their original question and is in fact false. In order to show this we will translate the problem into a statement about the vertical Euler class of a certain stable group.

To this end we consider the sequence of inclusions of mapping class groups given by adding an annulus with one marked point to a genus $h$ surface with one boundary component and $k$ marked points:
\[... \to \Gamma^1_{h,k - 1} \to \Gamma^1_{h, k } \to \Gamma^1_{h, k + 1}  \to ... \ .\]
The injective limit of this sequence will be denoted by $\Gamma^1_{h, \infty}$ and the vertical Euler classes on $\Gamma^1_{h, k} $ defines a vertical Euler class $e^v_{\infty}$ in $H^2(\Gamma^1_{h, \infty})$.
\begin{lem}\label{bounded_stable}
Suppose that the norms $||e^v_k||_{\infty}$ are bounded independently of $k$. Then the class $e^v_{\infty}$ is bounded.
\end{lem}
\begin{proof}
We let $C = \sup_k ||e^v_k||_{\infty}$. This also serves as a universal bound on the associated classes on $\Gamma^1_{h, k }$ and hence for any homology class $\alpha \in H_2(\Gamma^1_{h, \infty})$ one has
\begin{align*}
|e^v_{\infty}(\alpha)| = |e^v_{k}(\alpha_n)| & \leq ||e^v_k||_{\infty} \thinspace ||\alpha_k||_1\\
& \leq C ||\alpha_k||_1,
\end{align*}
where $\alpha_k$ is any class in $H^2(\Gamma^1_{h, k})$ which projects to $\alpha$ in the limit. By choosing the $\alpha_k$ appropriately we may assume that $||\alpha||_1 = \lim\limits_{k \to \infty} ||\alpha_k||_1$. Thus after applying limits to the inequality above we conclude that $|e^v_{\infty}(\alpha)| \leq  C ||\alpha||_1$ and hence that $e^v_{\infty}$ is a bounded class.
\end{proof}
In order to show that the class $e^v_{\infty}$ is not bounded, it will be important to have an explicit description of the classes $e^v_{k}$ on the level of group cohomology. To this end we let $\Sigma^{k,1}_h$ be a surface with $k + 1$ boundary components and we choose identifications $h_l: S^1 \to \partial \Sigma^{k,1}_h$ of the $l$-th boundary component with the circle for $1 \leq l \leq k$. We let $Diff^+(\Sigma^{k,1}_h)$ denote the group of all orientation preserving diffeomorphisms of $\Sigma^{k + 1}_h$ that fix the $(k+1)$-st boundary component pointwise and let $\sigma$ denote the natural map $Diff^+(\Sigma^{k,1}_h) \to S_k$ given by the action on the boundary components. We next consider the subgroup $\widehat{Diff^+}(\Sigma^{k,1}_h)$ of $Diff^+(\Sigma^{k,1}_h)$ defined by taking those elements $\phi$ that have the additional property that $\phi \circ h_l = h_{\sigma(l)}$ for all $1 \leq l \leq k$. The group of components of $\widehat{Diff^+}(\Sigma^{k,1}_h)$ will be denoted $\hat{\Gamma}^1_{h,k}$. This group fits into the following exact sequence, where the second map is given by coning off the boundary components and the kernel is generated by Dehn twists around parallels of the boundary components: 
\[ 1 \to \mathbb{Z}^k \to \hat{\Gamma}^1_{h,k} \to \Gamma^1_{h,k}  \to 1.\]
The conjugation action of $\Gamma^1_{h,k}$ on $\mathbb{Z}^k$ is given by the natural action of the symmetric group. Thus, this extension gives an element  $\hat{e}_k^v$ in cohomology with twisted coefficients $H^2(\Gamma^1_{h,k}, \mathbb{Z}^k)$, where the action is given via the natural map to $S_k$. The map $\mathbb{Z}^k \to \mathbb{Z}$ sending each basis vector to $1$ is $S_k$-equivariant and induces a map to ordinary cohomology so that the image of $\hat{e}_k^v$ is $e^v_k \in H^2(\Gamma^1_{h,k}, \mathbb{Z})$. The associated central $\mathbb{Z}$-extension will be denoted by $\overline{\Gamma}^1_{h,k}$. By considering the injective limits of the natural inclusions 
\[... \to \hat{\Gamma}^1_{h,k - 1} \to \hat{\Gamma}^1_{h, k} \to \hat{\Gamma}^1_{h, k + 1}  \to ... ,\]
we obtain a group $\hat{\Gamma}^1_{h,\infty}$ that fits into the following short exact sequence 
\[ 1 \to \mathbb{Z}^{\infty}  \to \hat{\Gamma}^1_{h,\infty} \to \Gamma^1_{h,\infty} \to 1.\]
This extension corresponds to an element  $\hat{e}_{\infty}^v \in H^2(\Gamma^1_{h,\infty}, \mathbb{Z}^{\infty})$, which denotes the cohomology group taken with twisted $\mathbb{Z}^{\infty}$-coefficients. As above, there is an $S_{\infty}$-equivariant map $\mathbb{Z}^{\infty} \to \mathbb{Z}$ sending each basis vector to $1$ that induces a map to ordinary cohomology and the image of $\hat{e}_{\infty}^v$ is the class $e^v_{\infty}$ defined above. The class $e^v_{\infty}$ then determines a central $\mathbb{Z}$-extension 
\[ 1 \to  \mathbb{Z}  \to \overline{\Gamma}^1_{h,{\infty}} \to \Gamma^1_{h,{\infty}} \to 1.\]
Furthermore, the group $\overline{\Gamma}^1_{h,{\infty}}$ fits into the following exact sequence
\[ 1 \to K  \to \hat{\Gamma}^1_{h,\infty} \to \overline{\Gamma}^1_{h,{\infty}} \to 1,\]
where $K$ is the kernel of the map $\mathbb{Z}^{\infty} \to \mathbb{Z}$.

The group $\hat{\Gamma}^1_{h,\infty}$ may also be interpreted as the mapping class group of a certain non-compact surface with infinitely many boundary components. More precisely, we let $\Sigma^{\infty,1}_{h}$ be a genus $h$ surface having an infinite end with infinitely many open discs removed and we identify the $l$-th boundary component of $\Sigma^{\infty,1}_{h}$ with $S^1$ via a map $h_l$. We let $Diff^c(\Sigma^{\infty,1}_{h})$ denote the group of compactly supported diffeomorphisms of $\Sigma^{\infty,1}_{h}$ and $\sigma$ the natural map $Diff^c(\Sigma^{\infty,1}_{h}) \to S_{\infty}$ given by the induced action on the boundary components. The group $\widehat{Diff^c}(\Sigma^{\infty,1}_{h})$ is then defined to be the subgroup of $\phi \in Diff^c(\Sigma^{\infty,1}_{h})$ such that $\phi \circ h_l = h_{\sigma(l)}$ and its group of components is naturally isomorphic to $\hat{\Gamma}^1_{h,\infty}$.

As a final preliminary we note that any homogeneous quasi-homomorphisms on $\hat{\Gamma}^1_{0,\infty}$ is in fact a homomorphism. The group $\hat{\Gamma}^1_{0,n} = \hat{B}_{n}$ is called the extended braid group since it fits into the following extension
\[ 1 \to  \mathbb{Z}^{n} \to \hat{\Gamma}^1_{0,{n}} \to \Gamma^1_{0,{n}} = B_{n} \to 1,\]
where $B_{n}$ denotes the ordinary braid group on $n$-strands. Kotschick has proven that the only homogeneous quasi-homomorphisms on $B_{\infty}$ are homomorphisms and the proof of the following proposition is almost identical to that of Theorem 3.5 in \cite{Kot3}.
\begin{prop}\label{stable_Kotschick}
Any homogeneous quasi-homomorphism on the group $\hat{B}_{\infty}$ is a homomorphism.
\end{prop}
\begin{proof}
By considering disjoint embeddings $\Sigma^{n,1}_0 \hookrightarrow \Sigma^{\infty,1}_0$ one obtains infinitely many embeddings $\hat{B}_{n} \hookrightarrow \hat{B}_{\infty}$ that are conjugate in $\hat{B}_{\infty}$ and such that any two elements that lie in distinct embeddings commute in $\hat{B}_{\infty}$. Thus, by Proposition 2.2 in \cite{Kot3} any homogeneous quasi-homomorphism $\phi$ on $\hat{B}_{\infty}$ restricts to a homomorphism on $\hat{B}_{n}$. Furthermore, any element in $\hat{B}_{\infty}$ lies in some $\hat{B}_{n}$ and, hence, $\phi$ is in fact a homomorphism on the whole group $\hat{B}_{\infty}$.
\end{proof}
\noindent We now come to the proof that the class $e^v_{\infty}$ is not bounded.
\begin{thm}\label{unbound}
The class $e^v_{\infty} \in H^2(\hat{\Gamma}^1_{h,\infty})$ is unbounded.
\end{thm}
\begin{proof}
We consider the central $\mathbb{Z}$-extension associated to $e^v_{\infty}$
\[ 1 \to  \mathbb{Z}  \to \overline{\Gamma}^1_{h,{\infty}} \to \Gamma^1_{h,{\infty}} \to 1,\]
whose kernel is generated by some $\Delta$. By Lemma \ref{bound_commutator}, $e^v_{\infty}$ is bounded if and only if the stable commutator length of $\Delta$ is positive. Hence, in order to show that the class $e^v_{\infty}$ is unbounded, it will suffice to show that the stable commutator length of $\Delta$ is zero.

The element $\Delta$ is the image of a Dehn twist around a boundary parallel curve in $\Sigma^{\infty,1}_{h}$ under the map $\hat{\Gamma}^1_{h,{\infty}} \to \overline{\Gamma}^1_{h,{\infty}}$. More generally, let $\gamma_k$ be a simple closed curve in $\Sigma^{\infty,1}_{h}$ that bounds $k$ boundary components and let $\Delta_k$ be a Dehn twist about $\gamma_k$. The stable commutator length of $\Delta_k$ depends only on $k$, since any two curves $\gamma_k$ and $\gamma'_k$ are conjugate by a diffeomorphism in $\widehat{Diff^c}(\Sigma^{\infty,1}_{h})$ so that the corresponding Dehn twists are conjugate in $\hat{\Gamma}^1_{h,\infty}$. In this notation $\Delta$ is the image of $\Delta_1$ under the map $\hat{\Gamma}^1_{h,{\infty}} \to \overline{\Gamma}^1_{h,{\infty}}$ and it follows that
\[|| \Delta ||_{com} \leq || \Delta_1 ||_{com}.\]
So it will be sufficient to show that $|| \Delta_1 ||_{com} = 0$.

We assume to the contrary that the stable commutator length of $\Delta_1$ is positive. We let $D_4 \hookrightarrow \Sigma^{\infty,1}_{h}$ be a disc with four smaller discs removed, the boundaries of which are mapped to boundary components of $\Sigma^{\infty,1}_{h}$. This inclusion gives a map of the compactly supported mapping class group of $D_4$ into  $\hat{\Gamma}^1_{h,\infty}$. We let $C_1,..,C_4$ be the interior boundary components of $D_4$. We also let $\Delta(i_1,..,i_m)$ denote the Dehn twist around an embedded closed curve containing $C_{i_1},.., C_{i_m}$. Then the following lantern relation holds in the compactly supported mapping class group of $D_4$ and, thus, also in $\hat{\Gamma}^1_{h,\infty}$:
\[\Delta(12) \Delta(23) \Delta(13) = \Delta(1) \Delta(2) \Delta(3) \Delta(123).\]
 We let $\phi$ be a homogeneous quasi-homomorphism on $\hat{\Gamma}^1_{h,\infty}$. An inclusion $\Sigma^{\infty,1}_{0} \hookrightarrow \Sigma^{\infty,1}_{h}$ induces an embedding $\hat{B}_{\infty}\hookrightarrow \hat{\Gamma}^1_{h,\infty}$ and by Proposition \ref{stable_Kotschick} any homogeneous quasi-homomorphism restricts to a homomorphism on $\hat{B}_{\infty}$. In particular, $\phi$ is a homomorphism on the normal subgroup generated by the $\Delta_k$. Then since a homogeneous quasi-homomorphism is constant on conjugacy classes the lantern relation implies that
\[\phi(\Delta_3) = 3\phi(\Delta_2) - 3 \phi(\Delta_1).\]
Similarly one has
\[\Delta(123) \Delta(234) \Delta(34) = \Delta(12) \Delta(3) \Delta(4) \Delta(1234)\]
and by the same reasoning as above it follows that
\[\phi(\Delta_4) = 6\phi(\Delta_2) - 8 \phi(\Delta_1).\]
Furthermore, the embedding $\Sigma^{\infty,1}_{h} \to \Sigma^{\infty,1}_{h}$ given by attaching a $k$-punctured disc to each boundary component induces a map on $\hat{\Gamma}^1_{h,\infty}$ that sends a Dehn twist around a boundary component to a Dehn twist about some $\gamma_k$ and, hence,
\begin{equation}\label{ineq_delta_0}
|| \Delta_k ||_{com} \leq || \Delta_1 ||_{com}.
\end{equation}
By Theorem \ref{length_Bavard} and the assumption that $|| \Delta_1 ||_{com}$ is positive, we may choose a homogeneous quasi-homomorphism $\phi$ such that
\begin{equation*}
\frac{\phi(\Delta_1)}{2D(\phi)} \geq \frac{15}{16}|| \Delta_1 ||_{com} > 0.
\end{equation*}
Since the stable commutator length of $\Delta_2$ is bounded from above by the stable commutator length of $\Delta_1$, by applying Theorem \ref{length_Bavard} once more we deduce that
\begin{equation}\label{ineq_delta_2}
|| \Delta_1 ||_{com} \geq \frac{\phi(\Delta_1)}{2D(\phi)} \geq \frac{15}{16}|| \Delta_1 ||_{com} \geq\frac{15}{16} \frac{\phi(\Delta_2)}{2D(\phi)}.
\end{equation}
Using inequalities (\ref{ineq_delta_0}) and (\ref{ineq_delta_2}) we then compute
\begin{align*}
|| \Delta_1 ||_{com} \geq || \Delta_4 ||_{com} & \geq  \frac{|\phi(\Delta_4)|}{2D(\phi)}\\
& = |6 \frac{\phi(\Delta_2)}{2D(\phi)} - 8 \frac{\phi(\Delta_1)}{2D(\phi)}  |\\
& = |\Bigl(2\frac{\phi(\Delta_1)}{2D(\phi)} - \frac{6}{16}\frac{\phi(\Delta_2)}{2D(\phi)}\Bigr) + 6 \Bigl(\frac{\phi(\Delta_1)}{2D(\phi)} - \frac{15}{16}\frac{\phi(\Delta_2)}{2D(\phi)}\Bigr)|  \\
& \geq 2\frac{\phi(\Delta_1)}{2D(\phi)} - \frac{6}{16}\frac{\phi(\Delta_2)}{2D(\phi)} \geq \frac{24}{16}|| \Delta_1 ||_{com}> || \Delta_1 ||_{com},
\end{align*}
which yields a contradiction.
\end{proof}
As an immediate consequence of Lemma \ref{bounded_stable} we have the following corollary.
\begin{cor}\label{stable_sequence}
The sequence $||e^v_k||_{\infty}$ is unbounded.
\end{cor}
\noindent In fact by Proposition \ref{Euler_bounded_k} we know that this sequence can grow at most linearly with $k$. It would be interesting to know the precise growth rate of the sequence $||e^v_k||_{\infty}$, in particular whether it is linear or not. Furthermore, by plugging in the definitions we compute:
\begin{align*}   \sup_k||e^v_k||_{\infty} & = \sup_k\frac{|e^v_k(\alpha)|}{||\alpha||_1} = \sup_k \lim\limits_{n \to \infty}\frac{n|e^v_k(\alpha)|}{4g_{min}(n \alpha) - 4} \\
&= \sup_k \lim\limits_{n \to \infty}\frac{2^n|e^v_k(\alpha)|}{4g_{min}(2^n \alpha) - 4} = \sup_{k} \left(\sup_{n}\frac{2^n|e^v_k(\alpha)|}{4g_{min}(2^n \alpha) - 4}\right)\\
& = \sup_{k,n}\frac{2^n|e^v_k(\alpha)|}{4g_{min}(2^n \alpha) - 4} \leq \sup_{k,n}\frac{|e^v_k(n\alpha)|}{4g_{min}(n \alpha) - 4},
\end{align*}
where inequality in the second line holds, since the sequence $2^{-n}(4g_{min}(2^n \alpha) - 4)$ is monotone decreasing. This means that for each natural number $N$ there exist multisections $S_{N}$ in some $\Sigma_h$-bundle such that
\[\frac{|[S_{N}]^2|}{2g_N-2} = N.\]
By contrast this quotient is always $h-1$ in Example \ref{sections_of_bundles}, meaning that there exist fairly exotic multisections that cannot be described by simple covering tricks.

The fact that the class $e^v_{\infty}$ is not bounded fits in with the observation that stable cohomology classes tend to be unbounded. In \cite{Kot3} it was shown that the space of homogeneous quasi-homomorphisms on the stable mapping class group is trivial. Furthermore, it is easy to see that the image of the comparison from bounded cohomology to ordinary cohomology is trivial for the stable mapping class group $\Gamma_{\infty}$.
\begin{prop}\label{comp_bound}
The image of the comparison map $H_b^*(\Gamma_{\infty}) \to H^*(\Gamma_{\infty})$ is trivial. In particular, $H_b^2(\Gamma_{\infty})=0$.
\end{prop}
\begin{proof}
For any $h$ and any $n$ there are inclusions $\Delta_n \hookrightarrow (\Gamma_h)^n \hookrightarrow \Gamma_{\infty}$, where $\Delta_n \cong \Gamma_h$ denotes the diagonal subgroup. Let $\alpha \in H^*( \Gamma_{\infty})$, then there is a class $\sigma \in H_*(\Gamma_h)$ on which $\alpha$ evaluates non-trivially. The image of $\sigma$ under the inclusion of the diagonal is $n\sigma$ and thus for all $n$
\[||n\sigma||_1 \leq ||\sigma||_1,\]
which implies $||\sigma||_1 = 0$. Hence, since $\alpha(\sigma) \neq 0$ and $\alpha(\sigma) \leq ||\alpha||_{\infty}||\sigma||_1$, we conclude that $\alpha$ is unbounded proving the first claim.

For the second we consider the Bavard exact sequence
\[0 \to H^1(\Gamma_{\infty}) \to \widetilde{QH}(\Gamma_{\infty}) \to H_b^2(\Gamma_{\infty}) \to H^2(\Gamma_{\infty}).\]
The group $\widetilde{QH}(\Gamma_{\infty})$ is trivial by \cite{Kot3} and as we saw above the image of the final map is trivial. We conclude that $H_b^2(\Gamma_{\infty})$ as claimed.
\end{proof}
\noindent Proposition \ref{comp_bound} suggests the conjecture that the bounded cohomology of $\Gamma_{\infty}$ is in fact trivial, although we do not have any more compelling evidence for believing this.
\begin{conj}
The bounded cohomology of $\Gamma_{\infty}$ is trivial.
\end{conj}


\begin{thebibliography}{}

\bibitem[BG]{BGh} J. Barges and E. Ghys,
\textit{Surfaces et cohomologie born\'ee},
Invent. Math. \textbf{92} (1988), 509--526.

\bibitem[Bav]{Bav} C. Bavard,
\emph{Longueur stable des commutateurs},
Enseign. Math. (2) \textbf{37} (1991), no. 1-2, 109--150.

\bibitem[Bott]{Bott} R. Bott,
\emph{Lectures on characteristic classes and foliations},
in Lectures on algebraic and differential topology, Springer Lecture Notes in Mathematics, \textbf{279}, 1972.

\bibitem[Bow]{Bow2} J. Bowden,
\emph{Two-dimensional foliations on four-manifolds},
PhD Thesis, Ludwig-Maximillians-Universit\"{a}t, 2010. Available at: \url{http://edoc.ub.uni-muenchen.de/12551/}.

\bibitem[Cal]{Cal} D. Calegari,
\emph{scl},
MSJ Memoirs, 20. Mathematical Society of Japan, Tokyo, 2009.

\bibitem[EE]{EE} C. J. Earle and J. Eells,
\emph{The diffeomorphism group of a compact Riemann surface},
Bull. Amer. Math. Soc. \textbf{73} (1967), 557--559.

\bibitem[EK]{EnK} H. Endo and D. Kotschick,
\emph{Bounded cohomology and non-uniform perfection of mapping class groups},
Invent. Math. \textbf{144} (2001), 169--175.

\bibitem[GS]{GS} R. Gompf and A. Stipsicz,
\emph{Kirby calculus and the topology of 4-manifolds},
Graduate Studies in Mathematics \textbf{20}, AMS, 1999.

\bibitem[Gro]{Gro} M. Gromov,
\emph{Volume and bounded cohomology},
Inst. Hautes \'Etudes Sci. Publ. Math. \textbf{56} (1982), 5--99.

\bibitem[HH]{HH} F. Hirzebruch and H. Hopf,
\emph{Felder von Fl\"achenelementen in 4-dimensionalen Mannigfaltigkeiten},
Math. Ann. \textbf{136} (1958), 156--127.

\bibitem[Iva1]{Iva} N. Ivanov,
\emph{Mapping Class Groups}, (electronic) \url{http://www.mth.msu.edu/ ~ivanov/indexmath.html}.

\bibitem[Kot1]{Kot1} D. Kotschick,
\textit{Signatures, monopoles and mapping class groups},
Math. Res. Letters \textbf{5} (1998), 227--234.

\bibitem[Kot2]{Kot3} D. Kotschick,
\emph{Stable length in stable groups}, Groups of Diffeomorphisms in honor of Shigeyuki Morita on the occasion of his 60th birthday, Advanced Studies in Pure Mathematics 52, Mathematical Society of Japan 2008, 401--413.

\bibitem[Mats]{Mats} Y. Matsushita,
\emph{Fields of 2-planes and two kinds of almost complex structures on compact 4-dimensional manifolds},
Math. Z. \textbf{207} (1991), 281--291.

\bibitem[Mil]{Mil} J. Milnor,
\emph{On the existence of a connection with curvature zero},
Comment. Math. Helv. \textbf{32} (1958), 215--223.

\bibitem[MV]{MV} Y. Mitsumatsu and E. Vogt,
\emph{Foliations and compact leaves on 4-manifolds I: Realization and self-intersection of compact leaves}, Groups of Diffeomorphisms in honor of Shigeyuki Morita on the occasion of his 60th birthday, Advanced Studies in Pure Mathematics 52, Mathematical Society of Japan 2008, 415--442.

\bibitem[Mor1]{Mor} S. Morita,
\emph{Characteristic classes of surface bundles},
Invent. Math. \textbf{90} (1987), 551--577.

\bibitem[Mor2]{Mor3} S. Morita, \emph{Characteristic classes of surface bundles and bounded cohomology},
A F\^{e}te of Topology, Academic Press, 1988, 233--258.

\bibitem[OS]{OS} P. Ozvath and Z. Sz\'abo, \emph{The symplectic Thom conjecture},
Ann. of  Math. (2) \textbf{151} (2000), no. 1, 93--124.

\bibitem[Th]{Th1} W. Thurston,
\emph{The theory of foliations of codimension greater than one},
Comment. Math. Helv. \textbf{49} (1974), 214--231.

\end{thebibliography}
\end{document}